\documentclass[final,leqno,onefignum,onetabnum]{siamltex1213}

\usepackage{amssymb,amsmath,mathrsfs}
\input{cf_macros.sty}

\title{Nonuniform sampling and multiscale computation\thanks{ The work of Engquist is supported in part by NSF grant DMS 12 - 17203 and the Institute for Engineering and Scientific Computing (ICES) at The University of Texas at Austin}}
\author{Bj\"{o}rn Engquist\thanks{Department of Mathematics, The University of Texas at Austin, Austin, TX 78712, USA } \and Christina Frederick\footnotemark[2]}

\begin{document}
\maketitle

\begin{abstract}  In homogenization theory and multiscale modeling, typical functions satisfy the scaling law $f^{\epsilon}(x) = f(x,x/\epsilon)$, where $f$ is periodic in the second variable and  $\epsilon$ is the smallest relevant wavelength, $0<\epsilon\ll1$. Our main result is a new $L^{2}$-stability estimate for the reconstruction of such bandlimited multiscale functions $f^{\epsilon}$ from periodic nonuniform samples. The goal of this paper is to demonstrate the close relation between and sampling strategies developed in information theory and computational grids in multiscale modeling. This connection is of much interest because numerical simulations often involve discretizations by means of sampling, and meshes are routinely designed using tools from information theory. The proposed sampling sets are of optimal rate according to the minimal sampling requirements of Landau \cite{Landau}. 

%When $f^{\epsilon}$ is also bandlimited, stable sampling occurs for a bunched sampling set that corresponds to a numerical mesh consisting of a macroscale grid with spacing $\Delta x>\epsilon$ coupled with local microscale grids with spacing $\delta x <\epsilon$. These grids are very well matched with grids used in practice, for example the heterogeneous multiscale method (HMM) \cite{Abdulle2012, E2003}.   
\end{abstract}



\pagestyle{myheadings}
\thispagestyle{plain}

\normalsize

%%%%%%%%%%%%%%%%%%%%%%%%%%%%%%%%%%%%%%%
% !TEX root =  masterdoc.tex

\section{Introduction}

Multiscale modeling and computation has recently been a very active research field. A major concern is that in direct numerical simulation the smallest important scales must be resolved over the length of the largest scales in each dimension. In applications such as material science, chemistry, fluid dynamics, and biology, the problems contain a large variety of scales. It is therefore necessary to design numerical methods that efficiently capture fine scale features on the underlying computational mesh.

Shannon's sampling theorem \cite{Shannon} is often cited in the numerical analysis of linear or nonlinear systems that are discretized on a uniform grid. In order to ensure that the solution to the discretized problem can be represented, the grid must be sufficiently dense. Shannon's theorem implies that if the size of the computational domain is 1 and $0<\epsilon\ll1$ is the smallest important wavelength, then at least 2$\epsilon^{-1}$ unknowns are required in each dimension. 

As a consequence, full resolution may require a prohibitively high computational cost. Many different numerical frameworks have been proposed to handle this problem, for example the heterogeneous multiscale method (HMM) \cite{Abdulle2012,E2003}. HMM provides a framework for capturing large scale features on coarse grids with spacing $\Delta x>\epsilon$ by incorporating local simulations on grids with much finer resolution $\delta x<\epsilon$.  A key observation in this paper is that the grids used in these multiscale methods are very well matched with nonuniform sampling strategies for bandlimited functions \cite{Behmard2002, Marvasti2001}. 

Our goal is to extend the previous numerical analysis to the discretization of systems involving functions of the type studied in homogenization theory and multiscale modeling. The emphasis here is to make a connection between the two fields by viewing this class of functions from the perspective of information theory. When the frequency components of a function lie in a set of disjoint intervals that are spaced $O(\epsilon^{-1})$ apart, it is known that unique recovery is guaranteed from a clustered samples with sufficient density. We interpret the sampling density requirements in terms of the grid spacing in a macro-micro coupled mesh. The main theoretical contribution of this paper is Theorem \ref{thm:main-samp}, which provides a new $L^{2}$ stability estimate for the reconstruction of functions with structured bandlimitation from periodic nonuniform samples.

In the next section we briefly present an example of a HMM scheme for differential equations that serves as the main motivation for this work. Section \ref{sec:samplingbg} is devoted to a brief background of relevant issues in sampling theory for bandlimited functions. In \S\ref{sec:mainresult} we place multiscale functions in this context and state the main theorem. The proposed sampling strategy for this class of functions is given in \S\ref{sec:nonunifsamp}, and in \S\ref{sec:proof} we prove the $L^{2}$ stability of the reconstruction. We conclude in \S\ref{sec:conclusion}.

%%%%%%%%%%%%%%%%%%%%%%%%%%%%%%%%%%
\section{Functions from multiscale computation}\label{sec:bgsetting}
In homogenization theory and in convergence analysis of HMM, the basic objects of study are functions that contain variations on multiple scales.  In this paper we represent the multiscale nature of a function using the superscript $\epsilon$, where $\epsilon$ is a small parameter that represents the ratio of scales in the problem. The next example demonstrates the role of multiscale functions in the numerical treatment of differential equations.

\begin{example}[HMM for highly oscillatory systems] \rm
Consider stiff ordinary differential equations (ODEs) of the form
\begin{align}
\frac{d u^{\epsilon}}{dt} &= g^{\epsilon}(u^{\epsilon}, t),\labeleq{msode}
\end{align}
where $u^{\epsilon}$ is a solution that oscillates on the time scale of $O({\epsilon})$, $0<\epsilon\ll1$. 

Assume that  as $\epsilon\rightarrow 0$, $u^{\epsilon}\rightarrow U\in C^{1}(\RR)$ and that $U$ is given by
\begin{align}
\frac{d}{dt}U = \bar{g}(U,t).\labeleq{effode}
\end{align}

This ``effective'' system can be solved using HMM \cite{Abdulle2012,Engquist2005} even if the form of $\bar{g}$ is not explicitly known. The right hand side of \refeq{effode} can be approximated using averaged solutions to the full system. Figure $\ref{fig:hmmode}$ represents an HMM-type scheme for approximating the solution $U$ of \refeq{effode}. The top directed axis represents the coarse grid that holds values of $U$. In the lower axis, local solutions to \refeq{msode} are computed using an initial condition determined by $U(t_{n})$. Then, $\bar{g}$ is evaluated by averaging the solutions with a compactly supported kernel.

\begin{figure}
\caption{{\bf HMM for ordinary differential equations.} The diagram represents the macro-micro coupling in an HMM scheme for ODEs given in \cite{Engquist2005}. Here, the solution is calculated on a local microscopic mesh in order to approximate the solution on the macroscopic grid.}\label{fig:hmmode}
\begin{center}
\begin{tikzpicture}

 \tikzstyle{ann} = [fill=white,font=\footnotesize,inner sep=1pt]
    \begin{scope}[thick,font=\scriptsize]
    % Axes:
    % Are simply drawn using line with the `->` option to make them arrows:
    % The main labels of the axes can be places using `node`s:
    \draw [<->] (-5,2) -- (2.5,2) node [above left]  {};
    \draw [<->] (-5,0) -- (2.5,0) node [above left]  {};
 	\draw [	->] (-4,1.5) -- (-4,.5) node [above left]  {};
	%\node at (-4.5, 1) {(1a)};
	%\node at (-3.4, .35) {(1b)};
	%\node at (-2.1, 3) {Step 2};
	\draw [	->] (-2.9,.5) -- (-3.75,1.75) node [above left]  {};
	%\node at (-3, 1.2) {(1c)};
   \draw [	->] (0,1.5) -- (0,.5)  node [above left]  {};
	\draw [	->] (1.1,.5) -- (.25,1.75) node [above left]  {};

    % Axes labels:
    % Are drawn using small lines and labeled with `node`s. The placement can be set using options
    \iffalse% Single
    % If you only want a single label per axis side:
    \draw (1,-3pt) -- (1,3pt)   node [above] {$1$};
    \draw (-1,-3pt) -- (-1,3pt) node [above] {$-1$};
    \else% Multiple
    % If you want labels at every unit step:
    \foreach \n in {-4,0}{%
        \draw (\n,1.7) -- (\n,2.3)   node  [above] {$ $};  
        \draw (\n,-.3) -- (\n,.3)   node [above] {$ $}; 
    
     	    \draw (\n+.1,-3pt) -- (\n+.1,3pt)   node [above] {};        
    		\draw (\n+.2,-3pt) -- (\n+.2,3pt)   node [above] {};  
    		\draw (\n+.3,-3pt) -- (\n+.3,3pt)   node [above] {};  
    		\draw (\n+.4,-3pt) -- (\n+.4,3pt)   node [above] {};  
    		\draw (\n+.5,-3pt) -- (\n+.5,3pt)   node [above] {}; 
    	    \draw (\n+.6,-3pt) -- (\n+.6,3pt)   node [above] {}; 
      		\draw (\n+.7,-3pt) -- (\n+.7,3pt)   node [above] {}; 
            \draw (\n+.8,-3pt) -- (\n+.8,3pt)   node [above] {}; 
           	\draw (\n+.9,-3pt) -- (\n+.9,3pt)   node [above] {}; 
            \draw (\n+1,-3pt) -- (\n+1,3pt)   node [above] {}; 
            \draw (\n+1.1,-3pt) -- (\n+1.1,3pt)   node [above] {}; 
            
    }
       \draw (-4,1.7) -- (-4,2.3)   node (t1) [above] {$ $};  
       \draw (0,1.7) -- (0,2.3)   node (t2) [above] {$$}; 
       \node at (-4.5, 2.5) {$T=t_{n}$};
       \node at (-2.6, -.5) {$t_{n+\eta}$};
       \node at (-4.5, -.5) {$T=t_{n}$}
       (t1) edge[pil,bend left=45] (t2); 
        \draw (-4,-.3) -- (-4,.3)   node [above] {$ $}; 
		\node at (.7, 2.5) {$T=t_{n+1}$};

 % We make a dummy figure to make everything look nice.

    \fi
    \end{scope}
    \end{tikzpicture}
    \end{center}

\end{figure}
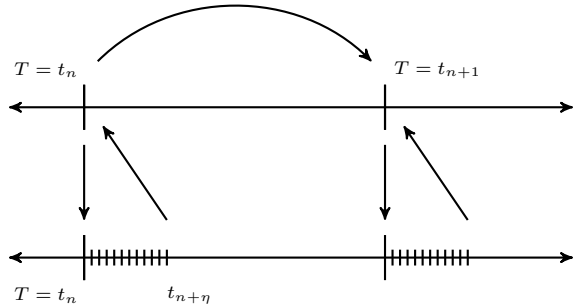

\end{example}

HMM captures the effective behavior of the system by exploiting known features of the solutions, such as scale separation or periodicity. A simple case is a model solution that varies on two distinct scales. The construction of a typical two-scale function $u^{\epsilon}$ begins with a ``slowly" varying function $u(x,y)$. A ``fast'' variable is introduced through the transformation $y\rightarrow x/\epsilon$, resulting in the representation
 \begin{align*}
u^{\epsilon}(x) = u(x, x/\epsilon) \quad \quad \text{where $ u(x,y)$ is periodic in $y$, } \quad 0 < \epsilon \ll1.
\end{align*}
The relevant functions are often of lower regularity, and in order to connect to information theory, we will approximate them by  bandlimited functions, described in the next section.

\section{Classical sampling theory for bandlimited functions}\label{sec:samplingbg}

We begin by defining the class of bandlimited functions.

\begin{definition}\rm
Let $\mc{F}\subset \RR^d$ be a bounded, measurable set. The space of $\mc{F}-$bandlimited functions is defined by \[\bl{\mc{F}}:=\{g\in L^{2}(\RR^d) \mid \hat{g}(\xi) = 0 \text{ for all } \xi \not \in \mc{F}\},\]
where the Fourier transform is given by $\hat{g}(\xi)=\int_{\RR^d} g(x)e^{2\pi i x \xi}dx$.
 \end{definition}
 
 The celebrated sampling theorem that Shannon used in his theory of communication \cite{Shannon1948, Shannon} provides a characterization of one-dimensional bandlimited signals. We shall now state the classical sampling theorem adapted to our mathematical framework.

\begin{theorem}[Classical sampling theorem]\label{thm:cst} Let $W>0$ and choose $\Delta x$ to be a fixed constant that satisfies $0<\Delta x\leq 1/2W $. Then,
\begin{align}
\forall f\in \bl{[-W, W]}, \quad f(x)&=\sum_{j=-\infty}^{\infty}f\(j\Delta x\)\frac{\sin\pi (x  - j\Delta x)}{\pi (x - n\Delta x)} \labeleq{rec},
\end{align}
where the convergence of the sum is in the $L^{2}-$norm and uniformly on $\RR$. 
\end{theorem}

We refer to the set $X=\{x_j\}$ as a sampling set, and the function values $f(y), y\in X$ as samples. When the adjacent points in $X$ are equidistant, $x_{j+1}-x_j = \Delta x$ for all $j$, then $X$ is called a uniform sampling set with sampling rate $\frac{1}{\Delta x}$. Theorem \ref{thm:cst} guarantees the recovery of a bandlimited function from its uniform samples provided that the sampling rate is greater than or equal to twice the highest frequency. The sampling rate $\frac{1}{\Delta x} = 2W$ is known as the \emph{Nyquist rate}.

Uniform sampling theory has been developed in more general contexts, including sampling of bandpass signals $f\in \bl{[-W+W_{0}, W+W_{0}]}$ \cite{Kohlenberg1953}, the $d-$dimensional uniform sampling theorem of Middleton and Peterson \cite{Peterson1962}, and the sampling theorem for locally compact abelian groups \cite{Kluvanek1965}.

In this work, we are interested in sampling signals that arise in important applications, including image processing, geophysics, and optical tomography. Due to the nature of the acquisition of measurements, it is not always possible to sample a function uniformly. Therefore, it is important to study the theory of \emph{nonuniform} or \emph{irregular} sampling.

Irregular sampling involves the reconstruction of a bandlimited function from nonuniformly spaced samples. For a review of the nonuniform sampling literature, see \cite{Marvasti2001}, and for theoretical and numerical aspects, see \cite{Feichtinger1992, Grochenig1992}. Of interest is \emph{periodic nonuniform sampling}, in which the sampling sets have the form $X=\cup_k X_k$, 
\begin{align}
X_{k}&:= \{j \Delta x + k\delta x \mid {j\in\ZZ^{d}} \}, \quad k\in \ZZ^d. \labeleq{samplingset}
\end{align}
This type of sampling is well studied in the signal processing literature \cite{Margolis2008, Prendergast, Vaidyanathan1998, Venkataramani2003, Yen1956}. One major challenge in the design of nonuniform sampling strategies from a practical point of view is the stability of the reconstruction \cite{Marks1985}.

\begin{definition}\rm Let $\mc{F}\subset \RR$ be a bounded, measurable set. For a given sampling strategy, a sampling set $X =\{x_{j}\}$ is a \textit{set of stable sampling}
for $\mathcal{B}(\mathcal{F})$ if there exists a constant $C>0$ such that
\begin{align*}
\int_{-\infty}^{\infty}|g(x)|^2dx \le C\sum_{j} |g(x_j)|^2 \text{ for all }  g\in
\mathcal{B}(\mathcal{F}).
\end{align*}
\end{definition}

Uniqueness and stability results for the nonuniform sampling of bandlimited functions are provided by Beurling and Landau  \cite{Beurling1989, Landau, Landau1967}. In these works, a bound is obtained on the minimum sampling density required for the stable reconstruction of a bandlimited function. 
\begin{definition}\rm
For a sampling set $X \subset \RR$,  the \emph{lower Beurling density} is determined by the sum of the bandwidths of a function, 
\begin{align*}
D^{-}(X) = \lim_{r\rightarrow \infty} \inf_{y\in \RR} \frac{|X\cap [y, y+r]|}{r},
\end{align*}
where the numerator is set to be the number of points in $X$ in each interval of length $r$.
\end{definition} 

Landau proved that if $X$ is a sampling set for a class of functions $f\in \bl{\mc{T}}$, $\mc{T}\subset \RR$, then  $D^{-}(X)\geq \lambda(\mc{T})$, where $\lambda$ is the Lebesgue measure. For certain multiband functions with large spectral gaps, this lower bound indicates the possibility of stable sampling at a rate that is much lower than the Nyquist rate.

There is an enormous volume of literature on sampling theory and its various generalizations and extensions \cite{Jerri1977, Papoulis1977a, Vaidyanathan2001}.  There are very similar nonuniform sampling theorems using bunched samples to reconstruct functions that have spectral gaps \cite{Behmard2002, Vaidyanathan1998, Venkataramani2000}. The result presented here differs from these results in terms of conditions on the spectral gaps or the notion of stability.

%%%%%%%%%%%%%%%%%%%%%%%%%%%%%%%%%%%%%%%

\section{Main result}\label{sec:mainresult}

The main theorem describes a sampling result for a class of multiscale functions with structured bandlimitation. Specifically, we assume that $f^{\epsilon}$ and $f$ are square-integrable functions satisfying
\begin{align}
f^{\epsilon}(x) &= f(x, x/\epsilon) \quad \quad \text{where $ f(x,y)$ is $Y$-periodic in $y$, }\quad 0 < \epsilon \ll1, \labeleq{twoscale}\\
\hat{f}(\xi) &= 0 \text{ for all } \xi \not \in  [-N, N]^d\times [-M, M]^{d};\quad\text{ $0<2N<\frac{1}{\epsilon}$, $M\geq 1$}, \labeleq{bl}  \end{align}
where $Y=[0,1]^{d}$. We will later show that $f^{\epsilon}$ contains spectral gaps with magnitude proportional to $\epsilon^{-1}$ and the Fourier transform satisfies
\begin{align}
\hat{f^{\epsilon}}(\xi) = 0 \text{ for all } \xi \not \in \bigcup_{|m|=0}^{M}\( [-N, N]^d+\frac{m}{\epsilon}\);\quad\text{ $0<2N<\frac{1}{\epsilon}$, $M\geq1$}. \labeleq{fepsspec}
\end{align} 
This specific structure will allow us to develop an optimal sampling strategy (see Figure \ref{fig:mb} for an example in one dimension). We will now state the main nonuniform sampling result for multiscale functions in the case $d=1$. The proof will be given in later sections. 

\begin{figure}[h!]
 \caption{{\bf Multiband spectrum and nonuniform sampling set ($d=1$).} The diagrams below represent the spectral support of a one-dimensional function satisfying \refeq{twoscale} and \refeq{bl} (top) and a periodic nonuniform sampling set with macroscale spacing $\Delta x$ and microscale spacing $\delta x$ (bottom). }\label{fig:mb}
\begin{center}
\begin{tikzpicture}
 \tikzstyle{ann} = [fill=white,font=\footnotesize,inner sep=1pt]
    \begin{scope}[thick,font=\scriptsize]
    % Axes:
    % Are simply drawn using line with the `->` option to make them arrows:
    % The main labels of the axes can be places using `node`s:
    \draw [<->] (-5,0) -- (5,0) node [above left]  {$\xi$};
    \draw [<->] (0,0) -- (0,1.5) node [above left]  {};

    % Axes labels:
    % Are drawn using small lines and labeled with `node`s. The placement can be set using options
    \iffalse% Single
    % If you only want a single label per axis side:
    \draw (1,-3pt) -- (1,3pt)   node [above] {$1$};
    \draw (-1,-3pt) -- (-1,3pt) node [above] {$-1$};
    \else% Multiple
    % If you want labels at every unit step:
    \draw[thick,fill=blue!20] (-.5,0) rectangle (.5,.4) node {$ $};
 
    \foreach \n in {-4,-2,0,2, 4}{%
       \draw[thick,fill=blue!20] (\n-.5,0) rectangle (\n+.5,.4) node [below] {$$};
     % \node[rectangle, draw, fill=blue!20, minimum size=1cm] at (\n-.5, 3pt) {$      $};
       
    }
   \node at (-3.0, .2) {${\bf \fontsize{5cm}{1em}\hdots}$};
   \node at (3.0, .2) {${\bf \fontsize{5cm}{1em}\hdots}$};
     \node at (0, -.5) {$[-N, N]    $};
	\node at (-4, -.5) {$[-N- \frac{M}{\epsilon}, N- \frac{M}{\epsilon}]$};
	\aa\node at (4, -.5) {$[-N+\frac{M}{\epsilon}, N+\frac{M}{\epsilon}] $};
	
    \fi
    \end{scope}
    \end{tikzpicture}
\\
\vspace{1cm}

\begin{tikzpicture}
 \tikzstyle{ann} = [fill=white,font=\footnotesize,inner sep=1pt]
    \begin{scope}[thick,font=\scriptsize]
    % Axes:
    % Are simply drawn using line with the `->` option to make them arrows:
    % The main labels of the axes can be places using `node`s:
    \draw [<->] (-5,0) -- (5,0) node [above left]  {};

    % Axes labels:
    % Are drawn using small lines and labeled with `node`s. The placement can be set using options
    \iffalse% Single
    % If you only want a single label per axis side:
    \draw (1,-3pt) -- (1,3pt)   node [above] {$1$};
    \draw (-1,-3pt) -- (-1,3pt) node [above] {$-1$};
    \else% Multiple
    % If you want labels at every unit step:
    \foreach \n in {-4.5,-2.5,-.5,1.5, 3.5}{%
        \draw (\n,-3pt) -- (\n,3pt)   node [above] {$ $};  
        \draw (\n+.3,-3pt) -- (\n+.3,3pt)   node [above] {}; 
        \draw (\n+.6,-3pt) -- (\n+.6,3pt)   node [above] {}; 
        \draw (\n+.9,-3pt) -- (\n+.9,3pt)   node [above] {};
       
    }
    \draw[decorate,decoration={brace}, thick](-4.5,10pt)--(-2.5,10pt);
  	\node[ann] at (-3.5,20pt) {$\Delta x$};
	\draw[decorate,decoration={brace,mirror}, thick](-4.5,-10pt)--(-4.2,-10pt);
  	\node[ann] at (-4.35,-20pt) {$\delta x$};
	
    \fi
    \end{scope}
    \end{tikzpicture}

\end{center}
\end{figure}
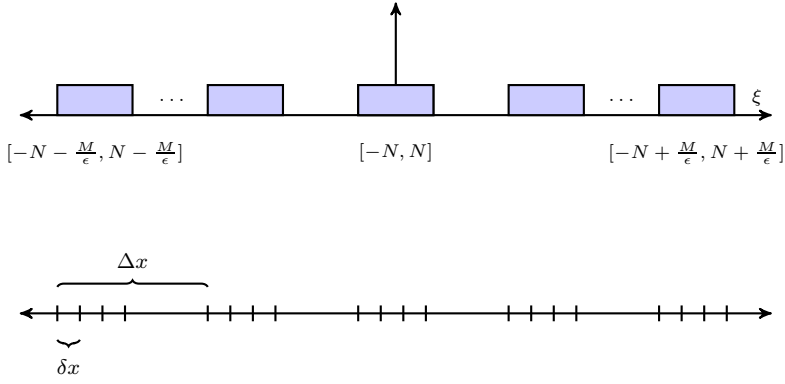

\begin{theorem}\label{thm:main-samp}
Let $f^{\epsilon}\in L^{2}(\RR)$ $(0<\epsilon\ll 1)$ be a function with Fourier transform $\hat{f^{\epsilon}}$ satisfying \refeq{fepsspec}. Define the nonuniform sampling set $X:=\cup_{k} X_{k}$, 
\begin{align}
X_{k}&:= \{j \Delta x + k\delta x \mid {j\in\ZZ} \}, \quad k=0,1,  \hdots P, \labeleq{samplingset1d}
\end{align} 
where the microscale spacing $\delta x$ and the macroscale spacing $\Delta x$ satisfy \[0<\delta x \leq \epsilon/(2M+1), \qquad \epsilon<{\Delta x} \leq{1}/{2N}.\]
If $P=2M$, then the function $f^\epsilon$ can be uniquely reconstructed from the samples $f^{\epsilon}(y)$, $y\in X$ and the following stability estimate holds:
\begin{align}
\|f^{\epsilon}\|_{L^{2}(\RR)}^{2} \leq C\sum_{y\in X}|f^{\epsilon}(y)|^{2},
\end{align}
 for a positive constant $C = C(\delta x / \epsilon)$,
 \begin{align}
 C(\delta x /\epsilon) =   \frac{1}{2N}{\( \sin{(\pi ({\epsilon^{-1}} - \Delta x^{-1})\delta x) }\)^{-2M}}.
\end{align}
\end{theorem}

If we set $\Delta x = \frac{1}{2N}$, then  $D^{-}(X)= \frac{2M+1}{\Delta x} = (2M+1)2N$. The computed lower Beurling density of of the set $X$ is therefore equal to the minimal average sampling rate for functions in the space $ \bl{ \displaystyle \cup_{m=-M}^{M} [-N+m/\epsilon, N+m/\epsilon]}$.

\begin{remark}\rm
 Multiscale problems can be divided into different groups based on common features of the problems, \cite{E2011}. 
The sampling analysis in this paper relates to, so called, type B problems. For this group microscale information is required throughout the entire computational domain.  Localized microscale simulations, spread all over the domain, supply missing information to the more efficient macroscale solver. Type A problems require microscale resolution with a microscale solver only in a fixed number of local domains. This could be in order to resolve isolated defects such as dislocations, cracks, and viscous shock profiles. Outside of these local domains, the macroscale solver is used. 

There are also links to information theory for type A problems involving functions of the following form
\begin{align}
f^{\epsilon}(t) =f(t, t/\epsilon), \qquad \underset{{|y|\rightarrow \infty}}{\lim}{f(t,y)} =
\begin{cases}
 f_{+}, &  t,y>0\\
 f_{-}, &  t,y<0,\\
 \end{cases}\labeleq{timewarped}.
\end{align}
This is essentially a smooth function $f^{-}$ that has a fast transition at zero and then is smooth again as $f^{+}$. Sampling theory for this type of warped signal is analyzed in \cite{Clark1985,Kempf2010}. A basic step in this analysis is determining an invertible transformation $\gamma$ such that $h^{\epsilon}(t):=f^{\epsilon}(\gamma(t))$ is a bandlimited function. The related sampling set $\{t_{j}\}$ for $h^{\epsilon}$ can be mapped back to the set $\{\gamma^{-1}(t_{j})\}$ that clusters points densely transition area.
\end{remark}

%%%%%%%%%%%%%%%%%%%%%%
\subsection{Multiband structure of $f^{\epsilon}$}\label{sec:ms}
In this section we derive the spectral property \refeq{fepsspec} for functions $f^\epsilon\in L^2(\RR^d)$ that satisfy scale separation \refeq{twoscale} and bandlimitation \refeq{bl}. We begin with an example from \cite{eng:sampling} of basic multiscale functions in one dimension that satisfy scale separation and periodicity.

\begin{example}[$d=1$]\label{ex:bl-periodic}
Let $f(x,y)\in L^2(\RR\times \RR)$ be function that is bandlimited and 1-periodic in \emph{both} variables. Then, $f(x,y)$ is represented by a finite Fourier series,
\begin{align}
f(x,y)=\sum_{n=0}^{N}\sum_{m=0}^M f_{n,m}e^{2\pi i(nx+my)}.\labeleq{ex:bl-periodic}
\end{align}
The multiscale function $f^{\epsilon}(x) = f(x,x/\epsilon)$ has Fourier transform
\begin{align*}
\hat{f^{\epsilon}}(k)=\sum_{n=0}^{N}\sum_{m=0}^M f_{n,m}\delta(k - n - {m}{\epsilon^{-1}}),
\end{align*}
where $\delta(x)$ is the Dirac delta function. Then, the support of the Fourier transform can be computed,
\begin{align*}
 \hat{f}^{\epsilon}(k) = 0, \quad k \not= n+ {m}{\epsilon^{-1}}, \qquad 0\leq n \leq N, 0\leq m\leq  M.
\end{align*}

\end{example}

Now, removing the assumption that $f(x,y)$ is periodic in $x$, tools from Fourier analysis can be applied to construct multiscale functions  from locally periodic functions of two variables. The next lemma provides the specific structure of the spectrum of bandlimited multiscale functions.

\begin{lemma}\label{lemma:feps-spec} Let $f^{\epsilon}$ and $f$ be square-integrable functions satisfying \refeq{twoscale} and \refeq{bl}. The Fourier transform $\hat{f^\epsilon}$ satisfies\begin{align}
\hat{f}^{\epsilon}(\xi)=0, \quad \xi \not \in \bigcup_{|m|=0}^M \left[-N+ \frac{m}{\epsilon}, N+ \frac{m}{\epsilon}\right]^d.\labeleq{lemfepsspec}
\end{align}

\end{lemma}
\begin{proof}
For every fixed $\x \in \RR^{d}$, the Fourier series expansion of the periodic function $f(x,\cdot)$ has the form
\begin{align}
f\left(\x, \y\right) =  \sum_{|m|=0}^M c_{m}\left(\x\right) e^{2\pi i\ip{m}{y}}; \quad c_{m}\left(\x\right)   =  \int_{Y} f\left(\x,\y\right)e^{-2\pi i\ip{m}{y}}dy.\labeleq{fs}
\end{align}
\noindent The sum is finite because $f$ is bandlimited in the second variable. It is readily seen that the Fourier coefficients $c_{m}(x)$ are $[-N, N]^{d}-$bandlimited functions,
 \begin{align*}
\hat{c}_{m}\left({\bf \xi}\right) %&=  \int_{\RR^{d}}\left(\int_Y f\left(\x,\y\right)e^{-2\pi i \ip{m}{y} }dy\right) e^{-2\pi i \ip{\xi}{x} }dx \\
 &=  \int_{Y}\int_{\RR^{d}} f\left(\x,\y\right)e^{-2\pi i \left(\ip{m}{y}+ \ip{\xi}{x}\right)}dxdy=  \hat{f}\left(\xi, m\right)= 0 \text{ for all } \xi \not \in [-N, N]^d.
\end{align*}

\noindent Substituting  $\left(\x,\frac{\x}{\epsilon}\right)$ for $\left(\x,\y\right)$  in \refeq{fs} results in the representation 
\begin{align}
f^{\epsilon}\left(\x\right)  &=  \sum_{|m|=0}^M c_{m}\left(\x\right)e^{2\pi i \ip{m}{x/\epsilon}}; \qquad c_{m}\in \bl{[-N, N]^d}.\labeleq{msrep}
\end{align}
For $\xi \in \RR^{d}$, the Fourier transform $\hat{f}(\xi)$ can be expressed as 
$
\hat{f^{\epsilon}}\left(\xi\right) = \sum\limits_{|m|=0}^M \hat{c}_{m}\left(\xi -\frac{m}{\epsilon}\right),
$ and since $\hat{c}_{m}(\xi-\frac{m}{\epsilon}) = 0$ for $\xi \not \in[-N+ \frac{m}{\epsilon}, N+ \frac{m}{\epsilon}]^d$, the result follows.
\end{proof}

 The terms with nonzero multi-indices in the union $\refeq{lemfepsspec}$ are created by shifting the set $[-N, N]^{d}$ by an integer multiple of $\frac{1}{\epsilon}$ in at least one dimension, and since it is assumed that $\frac{1}{\epsilon}-2N>0$, the shift results in a positive distance between the sets. This nonoverlapping property will allow for the design of nonuniform sampling schemes as in \cite{eng:sampling} for a broader class of functions. 

%%%%%%%%%%%%%%%%%%%%%%%%
\section{Nonuniform sampling strategy}\label{sec:nonunifsamp}

In this section we derive sufficient conditions on $\Delta x$ and $\delta x$ in order to ensure that a finite collection of sampling sets $X_{k}$,  defined by \refeq{samplingset}, can guarantee the recovery of functions $f^{\epsilon}$ with multiband structure \refeq{fepsspec}. A simple case is the class of periodic functions defined in Example \ref{ex:bl-periodic}.  

\begin{example}\rm
Let $f^\epsilon(x)=f(x,x/\epsilon)$ where $f(x,y)\in L^2(\RR\times \RR)$ is a bandlimited function that is periodic in both variables, and assume, for simplicity, that $\epsilon=1/L_1$ for a positive integer $L_1$. According to Theorem \ref{thm:cst},  the stable reconstruction of $f^\epsilon$ from uniform samples requires a sampling rate of $O(\epsilon^{-1})$. 

In the nonuniform sampling set used in \cite{eng:sampling}, the sampling points $\{x_{j,k}\}$  are clustered in groups with $\Delta x= {1}/{L_{2}}$, for ${L_{2}}$, ${L_1}/{L_{2}}$ positive integers,
\[x_{j,k}=j\Delta x +k \delta x\quad\quad 1\leq j \leq J, \quad 1\leq k \leq K \quad (\delta x < \Delta x).\]
Evaluating the expression \refeq{ex:bl-periodic} at each sampling point results in a system of equations
\[\sum_{n=0}^{N}\sum_{m=0}^M f_{n,m}e^{2\pi i(nx_{j,k}+mx_{j,k}/\epsilon)} = f^\epsilon(x_{j,k}).
\]
This system is invertible when $J>N$ and $K>M$ and the conditions $\delta x <\epsilon/M$ and $\Delta x <1/N$ are satisfied. This shows that it is possible to take advantage of the special structure of $f^\epsilon$ and uniquely reconstruct the function from samples taken from nonuniform sampling set with $O(N)$ density.

\end{example}

Now, removing the assumption that $f(x,y)$ is periodic in $x$, we develop a nonuniform sampling strategy. The proofs involve a modification of arguments from  \cite{Behmard2002}, where the lattice formed by the union of uniform sampling sets \refeq{samplingset} fails the required admissibility conditions.

We will first define a sampling operator that constructs functions using a Shannon-type reconstruction formula. For a function $g\in L^2(\RR^d)$, and a uniform sampling set of the form $X=\{x_0+ j\Delta x\mid j\in\ZZ^d\}$ for some $x_0\in \RR^d$, we formally define the sampling operator $S_{X}$ by
\begin{align}
S_{X}g(x)= \sum_{y\in X} g(y) \varphi_{s}(x-y)\labeleq{SX}, \qquad
\varphi_{s}(z) =  \frac{1}{\lambda(\spec_{s})}\int_{\spec_{s}} e^{2\pi i \ip{z}{\xi} }d\xi,\end{align}
where $\spec_{s}=[-\frac{1}{2\Delta x}, \frac{1}{2\Delta x}]^{d}$.

Due to higher dimensional generalizations of Shannon's sampling theorem \cite{Kluvanek1965, Peterson1962}, the sampling operator is well defined for $g\in \bl{[-N, N]^d}$ when $\Delta x \leq\frac{1}{2N}$, and in this case $S_Xg=g$. We show next that the operator is well defined for a class of functions that are undersampled by the set $X$, that is, the sampling rate is sub-Nyquist.

\begin{lemma}\label{lemma:SXkdef}
Let $f^{\epsilon}$ and $f$ be square-integrable functions satisfying \refeq{twoscale} and \refeq{bl}. Then, the function $S_{X_{k}} f^\epsilon$ corresponding to $X_k$, defined by \refeq{samplingset}, is square-integrable and satisfies
\begin{align}
S_{X_{k}}f^\epsilon(y) = f^\epsilon(y) \quad \text{ for all } y\in X_{k} \labeleq{Skg=g}.
\end{align}
\end{lemma}

\begin{proof}
Due to \refeq{msrep},
$f^\epsilon(x)=\sum_{|m|=0}^{M}c_{m}(x)e^{2\pi i \ip{m}{x/\epsilon}}$, where $c_{m}\in \bl{[-N, N]^d}$.
Applying the sampling operator to each $c_{m}$ results in
\begin{align}
 \(S_{X_{k}}c_{m}e^{2\pi i \ip{m}{\cdot/\epsilon}}\)(x) &= \sum_{y\in X_{k}} c_{m}(y)e^{2\pi i \ip{m}{y}/\epsilon} \varphi_{s}(x-y).\labeleq{SXkgm}
% &= \sum_{y\in X_{k}} c_{m}(y)e^{2\pi i \ip{L_{m}/\Delta x + \alpha_{m}}{y}} \varphi_{s}(x-y)\\
 %&= \sum_{y\in X_{k}} c_{m}(y)e^{2\pi i \ip{\alpha_{m}}{y}} \varphi_{s}(x-y) \nonumber
\end{align}
The function defined by $\tilde{c}_{m} = c_{m}(x+k\delta x)$ also lies in the space $\bl{[-N, N]^d}$, and we can bound the magnitude of the terms in the sum \refeq{SXkgm} by
\begin{align*}
 \sum_{y\in X_{k}} |c_{m}(y)e^{2\pi i \ip{m/\epsilon}{y}} \varphi_{s}(x-y)| &= \sum_{y\in X_{0}} |c_{m}(y+k\delta x)e^{2\pi i \ip{m/\epsilon}{y+k\delta x}} \varphi_{s}(x-y-k\delta x)| \\
  &\leq \(\sum_{y\in X_{0}} |\tilde{c}_{m}(y)|^{2} \)^{1/2}<\infty.
\end{align*}
The last statement holds because of the Cauchy Schwartz inequality and the square integrability of $\tilde{c}_{m}$. Therefore the series $\refeq{SXkgm}$ is uniformly absolutely-convergent and 
\begin{align*}
\| S_{X_{k}}c_{m}e^{2\pi i \ip{m}{\cdot/\epsilon}}\|_{L^{2}(\RR^{d})}^{2}\leq \(\sum_{y\in X_{0}} |\tilde{c}_{m}(y)|^{2} \)\|\varphi_{s}\|_{L^{2}(\RR^{d})}^{2}<\infty.
\end{align*}

\noindent Summing over $m$, 
\begin{align*}
{\sum_{|m|=0}^{M}} \(S_{X_{k}}c_{m}e^{2\pi i \ip{m}{\cdot/\epsilon}}\)(x)&= {\sum_{|m|=0}^{M}} \sum_{y\in X_{k}} c_{m}(y)e^{2\pi i \ip{m}{y}/\epsilon} \varphi_{s}(x-y)\\
&=  \sum_{y\in X_{k}}{\sum_{|m|=0}^{M}} c_{m}(y)e^{2\pi i \ip{m}{y}/\epsilon} \varphi_{s}(x-y)\\
&=\sum_{y\in X_{k}} f^\epsilon(y) \varphi_{s}(x-y) = S_{X_{k}}f^\epsilon(x).
\end{align*}
The function $S_{X_{k}}f^\epsilon(x)$ is a finite sum of square-integrable functions, and is therefore also square-integrable.

Then,  the statement \refeq{Skg=g} holds due to the calculations $\varphi_{s}(0) =  \frac{1}{\lambda(\spec_{s})}\int_{\spec_{s}} d\xi = 1$ and
$
\varphi_{s}(y) =  \frac{1}{\lambda(\spec_{s})}\int_{\spec_{s}} e^{2\pi i \ip{n\Delta x}{\xi} }d\xi
%\\
		   %&=  \frac{1}{|\spec_{s}|}\int_{\spec_{s}}\cos{2\pi  \ip{n\Delta x}{\xi} }+ i \sin{2\pi \ip{n\Delta x}{\xi} }d\xi,\\
		   =  \int_{[-\frac{1}{2}, \frac{1}{2}]^{d}}e^{2\pi i \ip{n}{\xi'} }d\xi' =  0
$
for $y = n\Delta x$, $n\not= 0$.

\end{proof}

The next lemma provides the explicit form of the function reconstructed from sub-Nyquist sampling of $f^{\epsilon}$. The assumption that the macroscale spacing is large with respect to the smallest scale, $\Delta x >\epsilon$, ensures the existence of unique constants $L_{m}\in \ZZ^d$ and $\alpha_{m}\in [0,1/\Delta x)^d$ that satisfy 
 \begin{align*}
 \frac{m}{\epsilon}= \frac{L_{m}}{\Delta x}+\alpha_{m}.
\end{align*}
This representation will allow us to allow for sampling rates $\frac{1}{\Delta x}$ that are not necessarily integer multiples of the highest frequencies.

\begin{lemma} \label{lemma:SX} Let $f^{\epsilon}$ satisfy the assumptions of Lemma \ref{lemma:SXkdef}. The function $S_{X_{k}}f^{\epsilon}$, $k\in \ZZ$ has the explicit form
\begin{align}
S_{X_{k}}f^{\epsilon}(x) =\sum_{|m|=0}^{M}{c}_{m}(x)e^{2\pi i (\ip{\alpha_{m} }{x}+\ip{L_{m}/\Delta x}{k\delta x})} .\labeleq{SXkfeps}
\end{align}
 Moreover, the reconstruction of $S_{X_{k}}f^{\epsilon}$ is stable, 
\begin{align}
\|S_{X_{k}}f^{\epsilon}\|_{L^{2}(\RR^d)}^{2}&=\Delta x^d  \sum_{y\in X_{k}}|f^{\epsilon}(y)|^{2}\leq \frac{1}{(2N)^d}  \sum_{y\in X_{k}}|f^{\epsilon}(y)|^{2}.\labeleq{SMbound}
\end{align}
\end{lemma}

\begin{proof}
In the proof of Lemma \ref{lemma:SXkdef}, it is shown that the function $S_{X_{k}}f^{\epsilon}$ is well defined, square-integrable, and can be expressed in terms of its sampled Fourier coefficients \begin{align}
S_{X_{k}}f^{\epsilon}(x) =\sum_{|m|=0}^{M}S_{X_{k}}{c}_{m}^{\epsilon}(x),\labeleq{SXkcm}
\end{align}
where ${c}_{m}^{\epsilon}(x)= c_{m}(x)e^{2\pi i \ip{m/\epsilon}{x}}$ is a function in the space $\bl{[-N+\frac{m}{\epsilon}, N+\frac{m}{\epsilon}]^d}$. Define the shifted function $d_{m}(x) = c_{m}(x+k\delta x)e^{2\pi i \ip{m/\epsilon}{k\delta x}}$. Then, for each $m$, 
\begin{align}
S_{X_{k}}{c}^{\epsilon}_{m}(x)&=\sum_{y\in X_{k}}c_{m}(y)e^{2\pi i \ip{m/\epsilon}{ y}}\varphi_{s}(x-y)\nonumber\\
%&=\sum_{y\in X_{0}}c_{m}(y+k\delta x)e^{2\pi i \ip{m/\epsilon}{ y+k\delta x}}\varphi_{s}(x-y-k\delta x)\nonumber\\
%&=\sum_{y\in X_{0}}d_{m}(y)e^{2\pi i \ip{\alpha_{m}}{ y}}\varphi_{s}(x-y-k\delta x)\nonumber\\
 &=\sum_{y\in X_{0}}d_{m}(y)e^{2\pi i \ip{\alpha_{m}}{ y}}\frac{1}{\lambda(\spec_s)}\int_{\spec_s}e^{2\pi i\ip{x- y-k\delta x}{\xi} }d\xi\nonumber\\
 &=\int_{\spec_s}\left[\frac{1}{\lambda(\spec_s)}\sum_{y\in X_{0}}d_{m}(y)e^{-2\pi i \ip{y}{ \xi-\alpha_{m}}}  \right] e^{2\pi i \ip{x-k\delta x}{\xi}}d\xi\nonumber\\
 &=\int_{\spec_s}\left[ \sum_{z\in\spec_{s}^{-1}}\hat{d}_{m}\(\xi- \alpha_{m} + z\)  \right] e^{2\pi i \ip{x-k\delta x}{\xi}}d\xi \labeleq{pois1}\\
&=\int_{\spec_s} \hat{d}_{m}(\xi- \alpha_{m} ) e^{2\pi i \ip{x-k\delta x}{\xi}}d\xi\nonumber\\
&= {d}_{m}(x-k\delta x) e^{2\pi i \ip{\alpha_{m} }{x-k\delta x}}\nonumber\\
&= {c}_{m}(x)e^{2\pi i (\ip{\alpha_{m} }{x}+\ip{L_{m}/\Delta x}{k\delta x})} \labeleq{cm}.
\end{align}
Here, the set $\spec_{s}^{-1}=\{ l/\Delta x\mid l \in \ZZ^{d}\}$ is called the \emph{reciprocal lattice}.

Since the functions $c_{m}$ and $\varphi_{s}$ are both square integrable, the sums converge uniformly and the exchange between sum and integral is justified. The Poisson summation formula in $\RR^{d}$ is used for \refeq{pois1} \cite{Pinsky2002}. Then, substituting \refeq{cm} in \refeq{SXkcm} proves the reconstruction formula. For stability, 
\begin{align*}
\|S_{X_{k}}f^{\epsilon}\|_{L^{2}(\RR^d)}^{2} = \|\sum\limits_{y\in X_{k}} f^{\epsilon}(y) \varphi_{s}(\cdot-y)\|_{L^{2}(\RR^d)}^{2}=  \Delta x^{d} \sum\limits_{y\in X_{k}}|f^{\epsilon}(y)|^{2}.
\end{align*}
\end{proof}

%%%%%%%%%%%%%%%%%%%%%%%%%%%%%%%%%%%%%%%

The expression \refeq{SXkfeps} plays an important role in the proof of the main result. 

\subsection{Proof of Theorem \ref{thm:main-samp}}\label{sec:proof}
The stable reconstruction formula for multiscale functions $f^{\epsilon}$ is derived using an approach similar to \cite{eng:sampling}. In the proof, we will need to estimate of the norm of  the inverse of Vandermonde matrices. Gautschi proved in \cite{Gautschi1990} that for arbitrary $w_{l}\in \mathbb{C}$, with $w_{l}\neq w_{l'}$ if $l\neq l'$, there holds
\begin{align}
\max_{l}\prod_{l'\neq l}\frac{\max(1, |w_{l'}|)}{|w_{l} - w_{l'}|} \leq\|V^{-1}\|_{\infty}\leq \max_{ l}\prod_{l'\neq l}\frac{1+|w_{l'}|}{|w_{l} - w_{l'}|} ,\labeleq{gautschi}
\end{align}
where $V$ is a Vandermonde matrix with elements $w_{0}, \hdots w_{P}\in \mathbb{C}$. The upper bound is obtained if $w_{l} = |w_{l}|e^{i\theta}, l=0, \hdots, P $ for some fixed $\theta \in \RR$.

Now we will prove the main result, Theorem \ref{thm:main-samp}.
\begin{proof}
The set of equations from \refeq{SXkfeps} form the linear system
\begin{align*}
S_{X_{k}}f^{\epsilon}(x) =\sum_{m=-M}^M{c}^{\alpha}_{m}(x)e^{2\pi i \frac{L_{m}}{\Delta x}{k\delta x}}, \qquad k=0, 1, \hdots P.
\end{align*}
where $c^{\alpha}_{m}(x) = c_{m}(x)e^{2\pi i {\alpha_{m}}{x}}$. When $P=2M$, the corresponding Vandermonde matrix $V$ contains the elements $V_{mk} = w_{m-M}^{k}$ for $w_{m} = e^{2\pi i \frac{L_{m}}{\Delta x}\delta x}$. If $w_{-M}, \hdots w_{ M}$ are distinct elements, the system is invertible. 

First, $w_0=1$. Since $0<\frac{L_M \delta x}{\Delta x}\leq\frac{M\delta x}{\epsilon}\leq\frac{M}{2M+1}<\frac{1}{2}$, the elements $w_{1}, \hdots, w_{M}$ are distinct nodes distributed on the upper half plane of the unit circle, and the elements $w_{-M}, \hdots, w_{-1}$ are distinct nodes distributed on the lower half plane of the unit circle. This ensures the existence of $V^{-1}$. As a result, the reconstruction formula for $f^{\epsilon}$ is well defined:
 \begin{align}
f^{\epsilon}(x) &= \sum_{m=-M}^{M}c_{m}^{\alpha}(x)e^{2\pi i \frac{L_{m}}{\Delta x}x }\\
 &= \sum_{j=0}^{2M}\(\sum_{k=0}^{P}(V^{-1})_{jk}S_{X_{k}}f^{\epsilon}(x)\)e^{2\pi i \frac{L_{j}}{\Delta x}x }.
\end{align}

Stability is shown using some properties of Vandermonde matrices and \refeq{gautschi}. The upper bound on $|w_{l} - w_{l'}|$ can be computed
\begin{align}
|w_{l} - w_{l'}|\leq|e^{2\pi i M\delta x/\epsilon} - 1|<|e^{2\pi i \frac{M}{2M+1}} - 1|<{2}.
\end{align}

\noindent Now we compute the smallest distance between adjacent nodes. In the first case, for $-M\leq l<M-1$,
\begin{align*}
|w_{l+1} - w_{l}| =|e^{2\pi i (\frac{1}{\epsilon} - (\alpha_{l+1} - \alpha_{l}))\delta x} -1|>|e^{2\pi i (\frac{1}{\epsilon} -  \frac{1}{\Delta x})\delta x} -1|.
\end{align*}

\noindent The distance between adjacent nodes $w_{M}$ and $w_{-M}$ is
\begin{align*}
|w_{M} - w_{-M}|%&=|e^{2\pi i L_M\delta x/\epsilon}-e^{2\pi i L_{-M}\delta x/\epsilon}|\\
&=|e^{2\pi i 2M\delta x/\epsilon-(\alpha_M-\alpha_{-M})\delta x}-1|\\
&=|e^{2\pi i(1- 2M\delta x/\epsilon+(\alpha_M-\alpha_{-M})\delta x)}-1|\\
&> |e^{2\pi i (\frac{1}{\epsilon} - \frac{1}{\Delta x})\delta x } - 1|.
\end{align*}
\noindent Since ${\Delta x}> \epsilon$, the last term in both cases is nonzero. Then, we have the estimate
\begin{align*}
\frac{1}{2^{2M}}<\|V^{-1}\|_{\infty}\leq \(\frac{2}{|e^{2\pi i (\frac{1}{\epsilon} -  \frac{1}{\Delta x})\delta x } - 1|}\)^{2M} =  \frac{1}{\( \sin{(\pi (\frac{1}{\epsilon} -  \frac{1}{\Delta x})\delta x) }\)^{2M}}. 
\end{align*}
Since $(\frac{1}{\epsilon}- \frac{1}{\Delta x})\delta x <\frac{\delta x}{\epsilon}< \frac{1}{2M+1}<\frac{1}{2}$, the denominator is bounded away from zero.
A final stability estimate for the reconstruction of $f^{\epsilon}$ from sampling sets $X_{k}, k=0, \hdots P$ is
 \begin{align*}
\|f^{\epsilon}\|_{L^{2}}^{2} &\leq \|V^{-1}\|_{\infty} \|\sum_{k}S_{X_{k}}f^{\epsilon}(x)\|_{L^{2}}^{2} < C(\delta x /\epsilon) \sum_{y\in \cup_{k}X_{k}} |f^{\epsilon}(y) |^{2},
\end{align*}
where the stability constant is
\begin{align*}
C(\delta x /\epsilon) \leq \frac{1}{2N\( \sin{\(\pi \({\epsilon^{-1}} - {\Delta x^{-1}}\)\delta x\) }\)^{2M}}.
\end{align*}
\qquad\end{proof}

Therefore, reconstruction is more stable when the nodes $w_{m}$ are maximally spaced apart on the unit circle, which occurs when $\delta x/\epsilon$ is close to $\frac{1}{2M+1}$. 

\begin{remark}\rm
The results can easily be generalized to functions $f(x,y)$ that are bandlimited to a subset of $[-N, N]^d\times [-M, M]^d$. As an example, we consider a function with only two frequency bands.

Let $f^\epsilon\in L^2(\RR)$ be a multiscale function of the form
\begin{align*}
f^\epsilon(x) = c_0(x)+ c_1(x)e^{2\pi i x/\epsilon}; \qquad c_0, c_1\in \bl{[-N, N]},\end{align*}
where $\frac{\Delta x}{\epsilon}= L_{1}\in \ZZ$.
Then, applying the sampling operator for the sets $X_0:=\{j\Delta x\mid j \in \ZZ\}$ and $X_1:=\{j\Delta x+\delta x \mid j \in \ZZ\}$ results in the reconstructed functions
\begin{align*}
S_{X_0}f^\epsilon(x) &= c_0(x)+ c_1(x),\\
S_{X_1}f^\epsilon(x) &= c_0(x)+ c_1(x)e^{2\pi i \delta x/\epsilon}.
\end{align*}
We have a Vandermonde system with $V = \begin{pmatrix}1 & 1 \\ 1 & w_1\end{pmatrix}$, where $w_1 = e^{2\pi i \delta x/\epsilon}$. Taking the inverse results in
\begin{align*}
V^{-1}&=\frac{1}{w_1-1}\begin{pmatrix}w_1 & -1 \\ -1 & 1\end{pmatrix},\\
\|V^{-1}\|_{\infty} &= \frac{2}{|w_1-1|} = \frac{1}{\sin(\pi \delta x/\epsilon).}
\end{align*}
Therefore, 
\begin{align*}
\|f^{\epsilon}\|_{L^{2}}^{2} &\leq \| V^{-1}\|_\infty\|S_{X_0}f^\epsilon + S_{X_1}f^\epsilon \|_{L^{2}(\RR)}^{2}\\
&\leq\frac{1}{2N\sin(\pi \delta x/\epsilon)} \sum_{y\in{X_0\cup X_1}}|f^\epsilon(y)|^2.\\
\end{align*}
Then if $\delta x = \frac{1}{2\epsilon}$, the stability constant is equal to $C = \frac{1}{2N}$.
\end{remark}

\begin{remark} \rm
The uniform sampling theorem in higher dimensions is described in \cite{Peterson1962}, along with an optimal sampling rate. To the best knowledge of the authors, the theory for sub-Nyquist sampling of multiband functions in higher dimensions is incomplete. The lemmas in previous sections allow us to adapt Theorem \ref{thm:main-samp} for $d\geq 1$. 
In the case of $d=1$, we showed sufficient conditions on the sampling sets that resulted in an invertible Vandermonde matrix $V$. In higher dimensions, the fundamental theorem of algebra does not hold and the invertibility of the system is not guaranteed. Therefore, new theory is needed to determine whether the system produced from nonuniform periodic sampling in higher dimensions allows for the stable reconstruction of multiband functions.
\end{remark}

%%%%%%%%%%%%%%%%%%%%%%%%%%%%%%%%%%%%%%%
\section{Conclusions}\label{sec:conclusion}

In this paper we show the $L_{2}$-stability of periodic nonuniform sampling for a class of functions studied in homogenization theory and multiscale analysis. These functions are of the type $f^{\epsilon}(x) = f(x, x/\epsilon)$, where $f(x,y)$ is periodic in the second variable and the parameter $\epsilon$, $0<\epsilon\ll1$ represents the ratio of scales in the problem. We view these functions in the setting of information theory by making the further assumption that $f$ belongs to the class of bandlimited functions. 

Then, applying tools from sampling theory, it is shown that $f^{\epsilon}$ is uniquely determined by sampling sets of the form 
\begin{align*}
f^{\epsilon}(z),\quad z \in \{j\Delta x + k \delta x \mid j \in \ZZ, 0\leq k \leq P-1\},
\end{align*}
where $P$ is the number of frequency bands in the spectrum of $f^{\epsilon}$. This matches well with grids used in multiscale simulations of coupled models on different  scales. Here, a solver coupling the macro and micro scales would find a compromise between the effective solution on a macroscale grid of size $\Delta x$ and the full direct numerical simulation on a fine scale grid with spacing $\delta x <\Delta x$. The new stability estimate provides a guideline for choosing the macroscale spacing $\Delta x$ and micro scale spacing $\delta x$ so that the average sampling rate attains the minimal sampling rate of Landau \cite{Landau}.

%%%%%%%%%%%%%%%%%%%%%%%%%%%%%%%%%%%%%%

%%%%%%%%%%%%%%%%%%%%%%%%%%%%%%%%%%%%%%%
\bibliographystyle{siam}
\bibliography{Thesis}

\end{document}